\documentclass[12pt]{article}
\usepackage{calc}
\usepackage{amsmath,amssymb,amsbsy,amsfonts,amsthm,latexsym,
amsopn,amstext,amsxtra,euscript,amscd}
\newtheorem{theorem}{Theorem}
\newtheorem{lemma}[theorem]{Lemma}

\newcommand{\cP}{{\mathcal P}}
\newcommand{\cQ}{{\mathcal Q}}

\begin{document}
\title{Additive bases with coefficients of newforms}


\author{
{\textsc{Victor Cuauhtemoc Garc\'{i}a}}\\
\\ Departamento de Ciencias B\'asicas\\
Universidad Aut\'onoma Metropolitana - Azcapotzalco\\
C.P. 02200, CD.MX, M\'exico\\
email: {\tt vc.garci@gmail.com} \\
\medskip\\
{\textsc{Florin Nicolae}}\\
 \\Simion Stoilow Institute of Mathematics\\ of the
Romanian Academy\\ P.O.BOX 1-764\\ 
RO-014700 Bucharest\\
email: {\tt florin.nicolae@imar.ro}
}
\date{\today}
\pagenumbering{arabic}
\maketitle
%
%
\begin{abstract}
Let $f(z)=\sum_{n=1}^{\infty}a(n) e^{2\pi i nz}$ be a normalized Hecke eigenform in $S_{2k}^{\text{new}}(\Gamma_0(N))$  with integer Fourier coefficients. We prove that there exists a constant $C(f)>0$ such that  any integer is a sum of at most $C(f)$ coefficients $a(n) $. It holds $C(f)\ll_{\varepsilon,k}N^{\frac{6k-3}{16}+\varepsilon}$.



\end{abstract}

\section{Introduction}
The set $\{\tau(n)\mid n\geq 1\}$ of values of Ramanujan's function is an additive basis of the 
integers: 
any integer $Z$ can be written as
\[Z= \sum_{j=1}^{74000}\tau(n_j).\]
See \cite{GaraevGarciaKonyagin}, \cite{Snurnitsyn}. 
Here we prove a similar property for the Fourier coefficients of normalized Hecke eigenforms.
For integers $k\geq 1$, $N\geq 1$ denote by $S_{2k}^{\text{new}}(\Gamma_0(N))$ the space of newforms of weight $2k$ on $\Gamma_0(N)$.

\begin{theorem}\label{thm:main}
     Lef $f(z)=\sum_{n=1}^{\infty}a(n) e^{2\pi i nz}$ be a normalized Hecke eigenform 
     in $S_{2k}^{\text{new}}(\Gamma_0(N))$ with integer 
     Fourier coefficients. There exists a constant  $C(f)>0$ such that  any integer $Z$ is a sum  
  \begin{equation}\label{MainSum}
      Z=\sum_{j=1}^{\ell}a(n_j)
    \end{equation} 
     for some $\ell \le C(f)$ and integers $n_j \ll  |Z|^{\frac{2}{2k-1}} +1.$
		It holds 
		$$C(f)\ll_{\varepsilon,k}N^{\frac{6k-3}{16}+\varepsilon}. $$
\end{theorem}

Our method follows the idea of ~\cite{GaraevGarciaKonyagin} to connect the
solubility of~\eqref{MainSum} with the Waring--Goldbach problem. 
We use results of Ram Murty on oscillations of Fourier coefficients of newforms,  of Matom\"aki on signs of the coefficients, 
and of Hua on the Waring-Goldbach problem, which are stated in the second section. The third section contains the proof of the theorem.

\section{Lemmas}

We apply the following facts.
\begin{lemma}[Ram Murty]\label{lemma:MurtyDensity}
   Lef $f(z)=\sum_{n=1}^{\infty}a(n) e^{2\pi i nz}$ be a normalized Hecke eigenform in 
   $S_{2k}^{\text{new}}(\Gamma_0(N))$. For any $\varepsilon>0$ we have 
  \[
    |a(p)|>(\sqrt{2}-\varepsilon)p^{\frac{2k-1}{2}}
  \] 
for a positive density of primes $p$.	
\end{lemma}
\begin{proof} This was proved in \cite[Corollary 2]{Murty} for forms on the full modular group, but the statement is true 
also for forms on $\Gamma_0(N)$  \cite[Chapter 4, Theorem 8.6 (ii) with $m=1$, page 89]{MurtyMurty}.
\end{proof}
\begin{lemma}\label{lemma:firstnegative}
   Lef $f(z)=\sum_{n=1}^{\infty}a(n) e^{2\pi i nz}$ be a normalized Hecke eigenform in 
   $S_{2k}^{\text{new}}(\Gamma_0(N))$. Let $n_f$ be the smallest integer such that $a(n_f)<0$ and $(n_f,N)=1$. Then 
	$$ n_f\ll (4k^2N)^{\frac38},$$
	where the implied constant is absolute.
\end{lemma}
\begin{proof} 
   See \cite[Theorem 1]{Mato}. See also \cite{KS}, \cite{IKS}, \cite{KLSW}.
\end{proof}
\medskip

   Let $k\ge 1$ be an integer, $p$ a prime number, $\theta\geq 0$ the integer with $p^{\theta}\mid k$ and 
	$p^{\theta+1}\nmid k, $
  \begin{equation}\label{eq:gammadef}
    \gamma = \begin{cases}
           \theta + 2, & \textrm{if } p=2\, \textrm{and } \theta>0, \\
           \theta + 1, & \textrm{otherwise},
             \end{cases}
  \end{equation}
   
  \begin{equation}\label{eq:Kdef}
	K=\prod_{(p-1)\lvert k}p^{\gamma}.
  \end{equation}
\begin{lemma}[Hua]\label{lemma:Hua}
   Let $k\ge 1$ be an integer and $K$ as in~\eqref{eq:Kdef}.
   If
  \[
     s\ge \begin{cases}
             2^k, & \textrm{if }  k \le 10 \\
             2k^2(2\log k + \log\log k + \frac{5}{2}),& \textrm{if }  k > 10,
          \end{cases}
  \]  
   then for any $ Z\equiv s \pmod K$ the number $ I_s(Z)$ of solutions of the equation
  \begin{equation}\label{eq:Waring-Goldbach}
    p_1^k +\cdots + p_s^k=Z,  \quad \textrm{for primes } p_1,\ldots, p_s,
  \end{equation}
   satisfies the following asymptotic formula
  \begin{equation}\label{eq:asymp}
      I_s(Z)=\mathfrak{G}(Z)\frac{\Gamma^s({1}/{k})}{\Gamma(s/k)}\frac{Z^{\frac{s}{k}-1}}{\log^s Z}+
             O_{k,s}\left(\frac{Z^{\frac{s}{k}-1}}{\log^{s+1} Z}\log\log Z\right),
  \end{equation} 
    where $\mathfrak{G}(Z)$ is the singular series
  \[
    \mathfrak{G}(Z)=\sum_{q=1}^{\infty}\left(\sum_{(h,q)=1}
    \left(\sum_{(l,q)=1}e^{2 \pi i h \frac{l^k}{q}}\right)^s e^{2\pi i \frac{-h}{q}Z}\right), 
  \] 
   which is absolutely convergent and 
    there exist positive constants $A,B$ independent of $Z$ such that
  \[
    0<A \le \mathfrak{G}(Z) < B.
  \]
\end{lemma}  
\begin{proof}
  See~\cite[Theorem 11 and Theorem 12, pages 78 and 100 respectively]{Hua}.
\end{proof} 
   Kumchev and Wooley proved in \cite[Theorem 1]{KW} that equation~\eqref{eq:Waring-Goldbach} 
   has solution if $k$ is large and $s\ge (4k-2)\log k +k-7$.
\section{Proof of the theorem}
   We denote by $\cP$ the set of prime numbers which do not divide $N$ and $\cP(X)=\cP\cap [1,X]$. 
   By Lemma~\ref{lemma:firstnegative} let 
     $n_f>1$ be the smallest integer such that $a(n_f)<0$ and $(n_f,N)=1$.
   Let $M$ be a large parameter 
    and set
  \[
    \cP_0(M) =\cP \cap (n_f,M].
  \]
    We will say that $D_M\subset \cP_0(M)$ is an admissible subset of $\cP_0(M)$ if 
  \[
     \sum_{i=1}^{k}a(p_i)\neq \sum_{i=k+1}^{2k}a(p_i) 
	\] 
	  for any $ p_1,\ldots, p_{2k} \in D_M $ such that
  \[
     p_1 < \ldots < p_k, \quad p_{k+1} < \ldots < p_{2k}, \quad 
     (p_1, \ldots, p_k) \neq (p_{k+1}, \ldots, p_{2k}).
  \] 
    We prove that admissible sets exist. 
    The Ramanujan--Petersson conjecture, proved by Deligne \cite{Deligne}, states 
    that $|a(p)|\le 2 p^{\frac{2k-1}{2}}$ for
    any prime number $p$. Let $\cP'\subset \cP$ be the set of prime numbers such that
  \begin{equation}\label{eq:coeffs}
     p^{\frac{2k-1}{2}}<|a(p)|.
  \end{equation}
    From Lemma~\ref{lemma:MurtyDensity} it follows that
    there exists a constant $0< \alpha \le 1,$ which depends on $f,$  
    such that for $T$ large enough we have
  \begin{equation}
    \cP' \cap [1,T]= \alpha \pi(T)\left(1+o(1)\right), \label{eq:largecoeffs}
  \end{equation}
  where $\pi(T) $ is the prime counting function.
    Let $\ell_0 > 10 \log k$ be an integer. For any $1 \le i \le 2k$ let 
  \begin{equation} \label{def:A_i}
    A_{i}:=\cP' \cap \left[{2^{\ell_0i}},{2^{\ell_0i+1}}\right].
  \end{equation}  
    From ~\eqref{eq:largecoeffs} it follows that 
  \[
    |A_i |> \frac{\alpha}{2} \frac{2^{\ell_0i}}{ \log 2^{\ell_0i}},
  \] 
    so we can choose M and $\ell_0$ sufficiently large such that $A_i \subset \cP_0(M)$ and
   $ |A_i|>1$ {for any}  $1\le i \le 2k.$ Let $p_i \in A_i, \, 1\le i \le 2k $. 
   From ~\eqref{eq:coeffs} and the Ramanujan-Petersson-Deligne estimate $|a(p_i)|\le 2 p_i^{\frac{2k-1}{2}} $ it follows that   
  \begin{equation}\label{ineq:p_sequence}
    |a(p_1)| < \cdots < |a(p_k)|<|a(p_{k+1})|  < \cdots <|a(p_{2k})|.
  \end{equation}  
    The set  $\cQ:=\{p_1,\ldots,p_{2k}\}$ is an admissible subset of $\cP_0(M)$. Indeed, let $q_1, \ldots, q_{2k}\in \cQ$ be such that 
  \begin{equation}\label{eq:Q_assumptions}
     q_1 < \ldots < q_k, \quad q_{k+1} < \ldots < q_{2k}, \quad 
     (q_1, \ldots, q_k) \neq (q_{k+1}, \ldots, q_{2k}), 
  \end{equation}
    and 
  \begin{equation}\label{eq:absurdum}
    \sum_{i=1}^k a(q_i) =\sum_{i=k+1}^{2k} a(q_i). 
  \end{equation} 
    Let $t$ be the largest index $1\le t \le k$ such that
    $ q_t \neq q_{k+t}$. From ~\eqref{eq:absurdum} it follows that 
  \[
    \sum_{i=1}^t a(q_i) =\sum_{i=k+1}^{k+t} a(q_i),
  \]  
  \[
    a(q_{k+t})= \sum_{i=1}^t a(q_i) -\sum_{i=k+1}^{k+t-1} a(q_i). 
  \]
    From ~\eqref{ineq:p_sequence} and ~\eqref{eq:Q_assumptions} it follows that
  \[
    |a(q_1)| < \cdots <|a(q_t)|, \quad |a(q_{k+1})| < \cdots <|a(q_{k+t})|,
  \]
   hence 
\begin{align}\label{ineq:Coeffs}
    |a(q_{k+t})|&=\left|\sum_{i=1}^t a(q_i) -\sum_{i=k+1}^{k+t-1} a(q_i)\right| \le 
    \sum_{i=1}^t |a(q_i)| +\sum_{i=k+1}^{k+t-1} |a(q_i)| \nonumber \\
               & \le  k (|a(q_t)| + |a(q_{k+t-1})|).
  \end{align}
   Without loss of generality we can suppose that $q_t < q_{k+t}$. Let $1\le s\le 2k$ be such that $q_{k+t}=p_s$. 
   It holds that $ q_{t},q_{k+t-1}\le p_{s-1},$ and from ~\eqref{eq:coeffs}, \eqref{def:A_i} and the 
   Ramanujan-Petersson-Deligne  estimate
   $|a(p_{s-1})|\le 2p_{s-1}^{(2k-1)/2} $ it follows that
  \[
    k (|a(q_)| + |a(q_{k+t-1})|) \le 2k |a(p_{s-1})| \le 4k 2^{(\ell_0(s-1)+1)\frac{2k-1}{2}},
  \]
  \[
    |a(q_{k+t})|=|a(p_s)|\ge 2^{\ell_0 s \frac{2k-1}{2}}.
  \]
    From the above  estimates and~\eqref{ineq:Coeffs} it follows that
  \[
    2^{\ell_0 s (2k-1)/2} \le  |a(q_{k+t})|\le k (|a(q_t)| + |a(q_{k+t-1})|)  \le 4k 2^{(\ell_0(s-1)+1)\frac{2k-1}{2}},
  \]  
  hence
  $$
     \ell_0 \le  1+ \frac{\log 4k}{(2k-1)\log 2},
  $$  
   which contradicts the assumption $\ell_0 > 10 \log k$. So   $\cQ$ is an admissible 
subset of $\cP_0(M)$.  
\medskip 

    Let $\cP_0'(M) \subset \cP_0(M)$ be some addmisible subset with largest cardinality. 
    We prove that
  \begin{equation}\label{ineq:P'estimate} 
     2k \le |\cP_0'(M)| \ll_k M^{\frac{2k-1}{2k}}.
  \end{equation}
    Since the admissible subset $\cQ$ constructed above has $2k$ elements it follows that 
  \[
     2k \le |\cP_0'(M)|. 
  \]       
   Let   
  \[
     S_k:= \{a(p_1)+\cdots +a(p_k)\;:\; p_1< \ldots < p_k, \, p_i \in \cP'_0(M)  \}.
  \] 
    Given  $\lambda \in S_k,$ let $T(\lambda)$ be the number of solutions $(p_1,\ldots,p_k) $ of the equation
  \[
      a(p_1)+\cdots +a(p_k) = \lambda, \quad p_1< \ldots < p_k, \quad p_i \in \cP'_0(M).
  \]  
     It holds that
  \[
     \sum_{\lambda \in S_k}T(\lambda)  \gg_k |\cP'_0(M)|^k.
  \]
     The Cauchy--Schwarz inequality implies that
  \begin{equation}\label{eq:SizeP0}
     |\cP'_0(M)|^{2k}\ll_k \left( \sum_{\lambda \in S_k}T(\lambda) \right)^2 \le 
     |S_k| \sum_{\lambda \in S_k}T^2(\lambda).
  \end{equation}
    Note that $\sum_{\lambda \in S_k}T^2(\lambda)$ is the number of solutions of the equation
  \begin{equation}\label{eq:symT}
    a(p_1)+\cdots +a(p_k) = a(p_{k+1})+\cdots +a(p_{2k}), 
  \end{equation}
    with
  \[
      p_1< \ldots < p_k, \quad
      p_{k+1}< \ldots < p_{2k}, \quad p_i \in \cP'_0(M).
  \] 
    Since $\cP_0'(M)$ is admissible ~\eqref{eq:symT} holds only if 
    $(p_1,\ldots,p_k)=(p_{k+1},\ldots,p_{2k})$. From this and ~\eqref{eq:SizeP0} 
    it follows that 
   \begin{equation}\label{ineq:LowBoundSk}
     |\cP'_0(M)|^{k}\ll_k |S_k|.
   \end{equation}
     The  estimate $|a(p)|\le 2 p^{\frac{2k-1}{2}}$ implies  
   \[
      |S_k|\ll_k M^{\frac{2k-1}{2}},
   \]  
     so from ~\eqref{ineq:LowBoundSk} we have
   \[
     |\cP_0'(M)|\ll_k M^{\frac{2k-1}{2k}}
   \]  
    and \eqref{ineq:P'estimate} is proved.
\medskip \\
    Let $p\in \cP_0(M)\setminus \cP_0'(M)$.  We proceed as in~\cite[Page 39]{GaraevGarciaKonyagin} to prove that there exist 
		$p_1, \ldots, p_{2k-1}$
    in $\cP'_0(M)$ such that
   \[
     a(p)=\sum_{i=1}^{k}a(p_i) - \sum_{i=k+1}^{2k-1}a(p_i).
   \] 
    Indeed, 
    the maximality of $\cP'_0(M)$ implies that there exist $$q_1, \ldots, q_{2k}\in \cP'_0(M)\cup \{p\}$$
    such that
 \[
     \sum_{i=1}^{k}a(q_i)= \sum_{i=k+1}^{2k}a(q_i), 
   \] 
   \begin{equation}\label{q-assumptions}
      q_1 < \ldots < q_k, \quad q_{k+1} < \ldots < q_{2k}, \quad 
     (q_1, \ldots, q_k) \neq (q_{k+1}, \ldots, q_{2k}).
   \end{equation}
    Moreover, 
   \[
     p \in \{q_1, \ldots, q_{2k}\},
   \] 
    and, by ~\eqref{q-assumptions}, $p$ occurs at most twice in
    the sequece $q_1, \ldots, q_{2k}$. If $p$ occurs twice, then it appears in
    $q_1, \ldots, q_{k}$ and in $q_{k+1}, \ldots, q_{2k}$, thus
   \[
     \sum_{i=1}^{k-1}a(q'_i)= \sum_{i=k+1}^{2k-1}a(q'_i), 
   \]      
    for some $q'_1,\ldots, q'_{2k}$ in $\cP'_0(M)$ with
   \[
      q'_1 < \ldots < q'_k, \quad q'_{k+1} < \ldots < q'_{2k},\quad (q'_1, \ldots, q'_k) \neq (q'_{k+1}, \ldots, q'_{2k}).
   \] 
    This is impossible, since $\cP_0'(M)$ is admissible with at least $2k$ elements.\\
		Therefore, for any 
    $p\in \cP_0(M)\backslash \cP'_0(M)$ there exist $p_1, \ldots, p_{2k-1}$
    in $\cP'_0(M)$\\\ such that
   \[
     a(p)=\sum_{i=1}^{k}a(p_i) - \sum_{i=k+1}^{2k-1}a(p_i).
   \] 

    Multiplying by $a(p)$ and taking into account that $(p,p_i)=1$ we get
   \[
      a(p)^2=\sum_{i=1}^{k}a(pp_i) - \sum_{i=k+1}^{2k-1}a(pp_i),
   \]
	since the coefficients of $f$ are multiplicative.
    Subtracting $a(p^2)$ and applying the identity
    $p^{2k-1}=a(p)^2- a(p^2)$ which is satisfied by the coefficients of $f$ it follows that
   \begin{equation}\label{eq:PrimePower}
     p^{2k-1}=\sum_{i=1}^{k}a(pp_i) - \sum_{i=k+1}^{2k-1}a(pp_i) - a(p^2).
 \end{equation}
Let 
  \[
    s_0\ge \begin{cases}
          2^{2k-1}, & \textrm{if }  2k-1 \le 10 \\
          2(2k-1)^2(2\log (2k-1) + \log\log (2k-1) + \frac{5}{2}),& \textrm{if }  2k-1 > 10.
          \end{cases}
  \] 
    We prove that for $Z$ large there exist 
    $p_1,\ldots,p_{s_0} \in \cP_0(Z^{1/(2k-1)})\setminus \cP'_0(Z^{1/(2k-1)})$ such that 
  \[
    Z=p_1^{2k-1}+\cdots + p_{s_0}^{2k-1}. 
  \] 
    Let $Z_k:=Z^{\frac{1}{2k-1}}$ and
  \[
      K=\prod_{p-1 \lvert 2k-1} p^{\gamma},
  \] 
    with $\gamma$ defined as in~\eqref{eq:gammadef}.  Since $2k-1$ is odd, the only prime number $p$ with 
		$p-1 \lvert 2k-1 $ is $p=2$, and for $p=2$ we have $$\theta =0 , \quad \gamma=1,$$ hence $$K=2.$$
    By Lemma~\ref{lemma:Hua} there exists a positive constant $c_1=c_1(k)$ such that 
    for any $ Z\equiv s_0 \pmod 2,$ with $Z$ large the number of solutions $I_{s_0}(Z)$ 
    of 
  \begin{equation}\label{eq:WaringGoldbach}
    p_1^{2k-1}+\cdots + p_{s_0}^{2k-1}= Z 
  \end{equation}
    with $p_1,\ldots, p_{s_0} \in \cP_0(Z_k),$ satisfies
  \begin{equation}\label{ineq:I}
    I_{s_0}(Z)\ge c_1 \frac{Z^{\frac{s_0}{2k-1}-1}}{\log^{s_0}Z}.
  \end{equation}
    Now consider equation~\eqref{eq:WaringGoldbach}  with at least
    one $p_i \in \cP'_0(Z_k)$ and denote by $I'_{s_0}(Z)$ its number of solutions.
    $I'_{s_0}(Z)$ should be less than $s_0 I'_{s_0-1},$ where $I'_{s_0-1}$ denotes the number 
    of solutions of the equation
  \begin{equation}\label{eq:ShortWaringGoldbach}
    p_1^{2k-1}+\cdots + p_{s_0-1}^{2k-1}+ p_{s_0}^{2k-1}= Z,  
  \end{equation}
    with $p_1,\ldots, p_{s_0-1} \in \cP_0(Z_k),$ and $p_{s_0}\in \cP'_0(Z_k).$
    Note that
  \[
    I'_{s_0-1}=\sum_{p_{s_0} \in \cP'_0(Z_k)} I'_{s_0-1}(Z-p^{2k-1}_{s_0}),
  \] 
    where $I'_{s_0-1}(Z-p^{2k-1}_{s_0})$ denotes the number of solutions of~\eqref{eq:ShortWaringGoldbach}
    for $p_{s_0}$ given. Therefore we have
  \[
    I'_{s_0-1}  \le \max_{p_{s_0}\in \cP'_0(Z_k)} \left\{
           I'_{s_0-1}(Z-p^{2k-1}_{s_0}) \right\}
           \sum_{{p_{s_0} \in \cP'_0(Z_k)}} 1.
  \]   
     Afterwards, for some $p'_{s_0} \in \cP'_0(Z_k)$ we get
  \begin{equation}\label{eq:I'}
     I'_{s_0-1}  \le I'_{s_0-1}(Z-p'^{2k-1}_{s_0}) \lvert \cP'_0(Z_k)\lvert.
  \end{equation}
    In order to estimate $I'_{s_0-1}(Z-p'^{2k-1}_{s_0})$ we apply  Lemma~\eqref{lemma:Hua} 
    with $s_0-1$ variables. Recalling that $Z-p'^{2k-1}_{s_0} > n_f^{2k-1},$ we obtain
  \begin{equation}\label{eq:I'p'}
    I'_{s_0-1}(Z-p'^{2k-1}_{s_0})\ll 
      \frac{(Z-p'^{2k-1}_{s_0})^{\frac{s_0-1}{2k-1}-1}}{\log^{s_0-1}(Z-p'^{2k-1}_{s_0})}
     \ll \frac{Z^{\frac{s_0-1}{2k-1}-1}}{\log^{s_0-1}Z}.
  \end{equation}
    Combining equations~\eqref{eq:I'}, \eqref{eq:I'p'} and  estimate~\eqref{ineq:P'estimate} 
    we get
  \begin{equation}\label{ineq:I'}
    I'_{s_0-1} \ll \lvert \cP'_0(Z_k)\lvert \frac{Z^{\frac{s_0-1}{2k-1}-1}}{\log^{s_0-1}Z}
    \ll_k  \frac{Z^{\frac{s_0}{2k-1}-1-\frac{1}{2k(2k-1)}}}{\log^{s_0-1}Z}.
  \end{equation}
    The number of solutions for \eqref{eq:WaringGoldbach} with 
    $p_i \in \cP_0(Z^{1/(2k-1)})\backslash \cP'_0(Z^{1/(2k-1)})$ is equal to $I_{s_0}(Z)-I'_{s_0}(Z).$
    The  estimates~\eqref{ineq:I} and \eqref{ineq:I'} imply that
   \[
      I_{s_0}(Z)-I'_{s_0}(Z)\ge I_{s_0}(Z)-s_0I'_{s_0-1} \gg_k 
      \frac{Z^{\frac{s_0}{2k-1}-1}}{\log^{s_0}Z} \left(1  - 
      \frac{\log Z}{Z^{1/(2k(2k-1))}} \right).
   \]
     Therefore  equation~\eqref{eq:WaringGoldbach} is solvable for primes in 
     $\cP_0(Z^{1/(2k-1)})\backslash \cP'_0(Z^{1/(2k-1)}).$
     From this and (15) it follows that any large integer $Z$ with
     $ Z\equiv s_0 \pmod 2$ has a representation
   \[
      Z= \sum_{i=1}^{ks_0} a(n_i) - \sum_{j=1}^{ks_0} a(n_j )
   \]  
     for some integers $n_i,n_j \le Z^{2/(2k-1)}$ with $(n_f!N,n_i)=(n_f!N,n_j)=1.$ Note that
     $-Z$ has a similar representation. We also note that  any integer $Z_0$ can be represented as 
   \[
     Z_0= r_0 + Z
	 \]
		with  $ Z\equiv s_0 \pmod 2, \; 0\le r_0 < 2$, thus if $Z_0$ is large then
   \[
     Z_0= \sum_{i=1}^{ks_0} a(n_i) - \sum_{j=1}^{ks_0} a(n_j ) + 
     \underbrace{a(1)+\cdots + a(1)}_{r_0-\text{times}},
   \] 
	   since $a(1)=1$.
     Recall that $n_f$ satisfies $a(n_f)<0$. Let $C_0:=-a(n_f)$. We have
   \[
      C_{0}Z_0 = C_0\sum_{i=1}^{ks_0} a(n_i) + \sum_{j=1}^{ks_0} a(n_jn_f) + C_0r_0 a(1)
	 \] 
     with $0\le r_0 < K$.
     As above, we note that for $Z_1$ large enough there exist integers $Z_0$ and $0\le r_1 < C_0$
     such that
   \begin{equation}\label{eq:Nlarge}
     Z_1 = C_0Z_0 + r_1 = C_0\sum_{i=1}^{ks_0} a(n_i) + \sum_{j=1}^{ks_0} a(n_jn_f) + C_0r_0 a(1) +
         r_1 a(1).
   \end{equation}
     Therefore, any integer $Z$ with $ |Z|\ge T$ can be expressed as in ~\eqref{eq:Nlarge} with
     $r_0 \le 2,$  $r_1\le C_0.$ The number of summands $a(n) $ in~\eqref{eq:Nlarge} is 
   \[
      (C_0+1)ks_0 + C_0r_0 + r_1 . 
   \] 
     
     For integers $Z$ with  $ |Z|\le T$  let $n'$ be such that $2Z < a_{n'}$. It holds that $|Z-a(n')|> T,$
     so $Z-a(n')$ can be written in the form ~\eqref{eq:Nlarge}. Hence any integer $Z$ can be written in the form  
   \[
     Z=\sum_{j=1}^{\ell}a(n_j ),
   \]
     with $$\ell \le (C_0+1)ks_0 + C_0r_0 + r_1+1\le (C_0+1)ks_0 + 2C_0 +C_0 +1=$$
		$$=(1-a(n_f))ks_0 -3a(n_f) +1=-a(n_f)(ks_0+3)+ks_0 +1,$$
		 since  $r_0 \le 2$ and  $r_1\le C_0$.  The theorem is proved with 
		$$ C(f):=-a(n_f)(ks_0+3)+ks_0 +1.$$
	Since $s_0$ depends only on $k$ we have
	$$C(f)\ll_k|a(n_f)|. $$
	By the Ramanujan-Petersson-Deligne estimate we have
	$$|a(n_f)|\le d(n_f)n_f^{\frac{2k-1}{2}},$$
	where $d(\cdot)$ is the number of divisors function which satisfies 
	$$d(n)\ll_\varepsilon n^\varepsilon,$$
	so
	$$a(n_f)\ll _\varepsilon n_f^{\frac{2k-1}{2}+\varepsilon}.$$
  By Lemma~\ref{lemma:firstnegative} we have
	$$n_f\ll (4k^2N)^{\frac38}, $$
	hence 
	$$C(f)\ll_{\varepsilon,k}N^{\frac{6k-3}{16}+\varepsilon}. $$
		$\Box$


\newpage



\begin{thebibliography}{999}

\bibitem{Deligne} P.~Deligne,
        {\it La conjecture de Weil. I.} (French)
         Inst. Hautes \'Etudes Sci. Publ. Math. (1974), No.~43 , 273--307. 

\bibitem{GaraevGarciaKonyagin} M.~Garaev, V.~Garc\'{i}a and S.~Konyagin,
        {\it Waring's problem with the Ramanujan $\tau-$function,}
        Russian Acad. Sci. Izv. Math. {\bf 72} (1) (2008), 45--46.

\bibitem{Hua} L.~K.~Hua, 
        {\it Additive theory of prime numbers,} Translations of Mathematical Monographs, 
        Vol. 13 American Mathematical Society, Providence,  1965.

\bibitem{IKS} H. Iwaniec, W. Kohnen, J. Sengupta, 
        {\it The first negative Hecke eigenvalue,}
        Int. J. Number Theory  {\bf 3} (2007), 355--363.

\bibitem{KS}  W. Kohnen, J. Sengupta, 
        {\it On the first sign change of Hecke eigenvalues of newforms,}
        Math. Z.  {\bf 254}(1) (2006), 173--184.

\bibitem{KLSW}  E. Kowalski, Y.-K. Lau, K. Soundararajan, J. Wu, 
        {\it On modular signs,}
       Math. Proc. Camb. Phil. Soc. {\bf 149} (3) (2010), 389--411.  

\bibitem  {KW} A. V. Kumchev, T. D. Wooley,  
            {\it On the Waring-Goldbach problem for eight and higher powers,}
        J. Lond. Math. Soc. (2) {\bf 93} (2016), no. 3, 811--824.

\bibitem  {Mato} K.~Matom\"aki, 
            {\it On signs of Fourier coefficients of cusp forms,}
        Math. Proc. Camb. Phil. Soc. {\bf 152} (2012), 207--222.
				 
\bibitem{Murty} M. Ram~Murty, 
        {\it Oscillations of Fourier coefficients of modular forms,}
        Math. Ann. {\bf 262} (1983), No.~4, 431--446.

\bibitem{MurtyMurty} M. Ram~Murty, V. Kumar~Murty, 
        {\it Non-vanishing of L-Functions and Applications,}
        Progress in mathematics, Vol. 157, Birkh\"auser, 1997.    

 
       
\bibitem{Snurnitsyn} P.~Snurnitsyn,
        {\it On basic properties of the Ramanujan $\tau-$function}
        Math. Notes, {\bf 90} (2011), No.~5, 736--743.         
 
\end{thebibliography}
\end{document}